\documentclass[12pt,reqno]{amsart}
\usepackage{amssymb}
\usepackage{amscd}
\usepackage{hyperref}

\newcommand{\func}[1]{\operatorname{#1}}
\newtheorem{theorem}{Theorem}[section]
\newtheorem{corollary}[theorem]{Corollary}
\newtheorem{lemma}[theorem]{Lemma}
\newtheorem{proposition}[theorem]{Proposition}

\newtheorem{definition}[theorem]{Definition}
\newtheorem{remark}[theorem]{Remark}

\numberwithin{equation}{section}
\begin{document}
\title{Homotopy invariance of cohomology and signature of a Riemannian
foliation}
\author[G.~Habib]{Georges Habib}
\address{Lebanese University \\
Faculty of Sciences II \\
Department of Mathematics\\
P.O. Box 90656 Fanar-Matn \\
Lebanon}
\email[G.~Habib]{ghabib@ul.edu.lb}
\author[K.~Richardson]{Ken Richardson}
\address{Department of Mathematics \\
Texas Christian University \\
Fort Worth, Texas 76129, USA}
\email[K.~Richardson]{k.richardson@tcu.edu}
\thanks{This work was supported by a grant from the Simons Foundation (Grant
Number 245818 to Ken Richardson), the Alexander von Humboldt Foundation,
Institut f\"ur Mathematik der Universit\"at Potsdam, and Centro
Internazionale per la Ricerca Matematica (CIRM)}

\begin{abstract}
We prove that any smooth foliation that admits a Riemannian foliation
structure has a well-defined basic signature, and this geometrically defined
invariant is actually a foliated homotopy invariant. We also show that
foliated homotopic maps between Riemannian foliations induce isomorphic maps
on basic Lichnerowicz cohomology, and that the \'Alvarez class of a
Riemannian foliation is invariant under foliated homotopy equivalence.
\end{abstract}

\maketitle
\tableofcontents


\section{Introduction}

One of the interesting problems of the theory of foliations is to compute
the basic index of a transverse Dirac-type operator in terms of topological
invariants, a generalization of the Atiyah-Singer theorem. This question,
which was first addressed by A. El Kacimi (see \cite[Problem 2.8.9]{EK}) and
by F.W. Kamber and J. Glazebrook (see \cite{GlazKamb}) in the 1980's, has
attracted significant attention by researchers during the last decades and
was open for many years. In order that such an index be well defined and
finite, we restrict to the class of foliations where the normal bundle is
endowed with a holonomy-invariant Riemannian structure, the setting of
Riemannian foliations; these were first defined in \cite{Re}, and good
information on these foliations and their analytic and geometric properties
can be found in \cite{To} and \cite{Mo}. On any such Riemannian foliation, a
so-called \emph{bundle-like metric} can be chosen on the whole manifold that
restricts to the given transverse metric on the normal bundle. For such a
metric, the leaves of the foliation are locally equidistant.

We are particularly interested in the basic signature operator, a
transversal version of the ordinary signature operator on even-dimensional
manifolds. Several results have been obtained in this direction. J. Lott and
A. Gorokhovsky in \cite{GorLott} state a formula under some conditions
involving the stratification and leaf closures of the foliation. As a
special case, they get an application to the basic signature operator,
showing that the basic signature of the foliation is the same as the
signature of the space of leaf closures of maximal orbit type, again under
various conditions. In the paper \cite{BKRfol} of J. Br\"{u}ning, F.W.
Kamber and K. Richardson, the authors obtain a general formula of the basic
index of a transversally elliptic operator on a Riemannian foliation. Using
these previous results, it is clear that the basic signature operator is
defined as a Fredholm operator on the space of basic sections of a foliated
vector bundle, and thus its index is dependent only on the homotopy class of
the principal transverse symbol of the operator. However, this type of
homotopy invariance, which is used in Proposition~\ref{BscSignIsInvtProp},
is a weaker special case of foliated homotopy
equivalence, which is simply a homotopy equivalence between foliations where
leaves get mapped to leaves. In this paper we discuss a much more
transparent description of the general homotopy invariance of the basic
signature of a Riemannian foliation.

In what follows, we remark that we are studying properties of operators on
basic forms, those differential forms that in a sense are constant on the
leaves of the foliation. The basic forms are forms in the transverse
variables alone when restricted to distinguished foliation charts. The
exterior derivative maps basic forms to themselves, and from this
differential we construct the basic cohomology groups. The basic signature
pairing is a pairing on the half-dimensional cohomology, similar to the case
of the ordinary signature of smooth manifolds.

In \cite{BenRA}, M. Benameur and A. Rey-Alcantara proved directly that a
foliated homotopy equivalence between two closed manifolds $M$ and $%
M^{\prime }$ endowed respectively with taut Riemannian foliations $\mathcal{F%
}$ and $\mathcal{F}^{\prime }$ implies that the corresponding basic
signatures are the same. The tautness assumption (also called \emph{%
homologically orientable} in some places) means that there exists a metric
for which the leaves are immersed minimal submanifolds. One main idea of the
proof was that any such equivalence induces an isomorphism between the
corresponding basic cohomology groups of $M$ and $M^{\prime }$. One
important observation is that in order to make this standard version of
basic signature on cohomology well-defined, the tautness assumption is
required, because in general the corresponding de Rham operator does not map
self-dual to anti-self-dual basic forms. On a Riemannian manifold $(M,g)$
endowed with a Riemannian foliation where the leaves are not necessarily
minimal, the authors in \cite{HabRic2} defined the basic signature operator
in terms of the index of the so-called twisted Hodge -- de Rham operator and
the twisted basic Laplacian. These latter operators are formed using the
twisted exterior differential $\widetilde{d}=d-\frac{1}{2}\kappa _{b}\wedge $%
, where $\kappa _{b}$ is the projection of the mean curvature one-form to
basic forms (see Section \ref{preliminariesSection} for details). From \cite%
{Al} it is well-known that $\kappa_b$ is always closed and determines a
class $[\kappa_b]$ (\emph{the \'Alvarez class}) in basic cohomology that is
independent of the bundle-like metric and of the Riemannian foliation
structure. We point out here that our definition was not possible for the
ordinary basic Laplacian since it does not commute with the transverse Hodge
star operator. What is interesting here is that the whole bundle-like metric
is used to form these operators and cohomology groups and classes, but as we
will soon see, the dimensions and indices coming from these groups and
operators do not depend on the metric choices. In this paper, our aim is to
understand and elucidate properties of the basic signature on all Riemannian
foliations on closed manifolds, in fact on all foliations admitting such
structures.

The paper is organized as follows. We first introduce the terminology
concerning Riemannian foliations, basic cohomology, twisted basic cohomology
and basic signature in Section \ref{preliminariesSection}. We discuss known
results from \cite{BenRA} concerning the homotopy invariance of ordinary
basic cohomology and the homotopy invariance of basic signature in the case
of taut Riemannian foliations.

Several new ideas are presented in Section \ref%
{twistedBscCohomBscSignSubsection}. In Theorem \ref{PairingThm} we prove
that there is a well-defined pairing in twisted basic cohomology $\widetilde{%
H}^{r}(M,\mathcal{F})\times \widetilde{H}^{s}(M,\mathcal{F})\rightarrow
H_{d}^{q-r-s}(M,\mathcal{F})$ given by $\left( \left[ \alpha \right] ,\left[
\beta \right] \right) \mapsto \overline{\ast }\left[ \alpha \wedge \beta %
\right] $, and this feature allows us to define a signature pairing when $%
r=s=\frac{q}{2}$ : 
\begin{equation*}
\left( \left[ \alpha \right] ,\left[ \beta \right] \right) \mapsto \int_M
\alpha \wedge \beta \wedge \chi _{\mathcal{F}}.
\end{equation*}%
Neither of these pairings would make any sense in ordinary basic cohomology
unless $\left( M,\mathcal{F}\right) $ is taut, but by using twisted basic
cohomology, the pairing is well-defined on all Riemannian foliations. Note
that the definitions of both twisted basic cohomology and the signature
pairing require use of a given bundle-like metric, but in Proposition \ref%
{BscSignIsInvtProp} we show that the invariants of the signature pairing are
actually smooth foliation invariants. In \cite{HabRic2} it was shown already
that the dimensions of the twisted basic cohomology groups are independent
of the metric or transverse structure of the foliation.

In Section \ref{BscLichnCohomHomotpyInvSubsection}, we prove properties of
basic Lichnerowicz cohomology, which was studied previously in \cite{Va}, 
\cite{Ban}, \cite{IdaP}, \cite{AitH}, \cite{OrnSle} and only uses the smooth
structure of the foliation. Given a closed basic one-form $\theta $, the map 
$d+\theta \wedge $ acts as a differential on basic forms and thus yields
cohomology groups $H_{d+\theta \wedge }^{\ast }\left( M,\mathcal{F}\right) $%
. The twisted basic cohomology discussed above is a special case of this
with $\theta =-\frac{1}{2}\kappa _{b}$. In Corollary \ref%
{HomotopyInvCorollary}, we show that foliated homotopies induce equivalent
maps on basic Lichnerowicz cohomology --- by \textquotedblleft
equivalent\textquotedblright\ we mean the same map up to multiplication by a
positive basic function. In Proposition \ref{FolHomotpEquivIsoProp}, we
prove that foliated homotopy equivalences induce isomorphisms on basic
Lichnerowicz cohomology. The importance of using the Lichnerowicz cohomology
is that we are able to use all possible closed one forms at once, and this
insight leads to the results in Section \ref{InvBscSignSubsection}.

In Proposition \ref{codimensionProp} we immediately use the Lichnerowicz
cohomology to prove easily that the codimension and dimension of a foliation
are foliated homotopy invariants. In Proposition \ref%
{PoincareDualityTwistedLichnProp}, we show that for a transversely oriented
Riemannian foliation $\left( M,\mathcal{F}\right) $ of codimension $q$,
basic Lichnerowicz cohomology satisfies twisted Poincar\'{e} duality, namely
that for $0\leq k\leq q$,%
\begin{equation*}
H_{d-\theta \wedge }^{k}\left( M,\mathcal{F}\right) \cong H_{d-\left( \kappa
_{b}-\theta \right) \wedge }^{q-k}\left( M,\mathcal{F}\right) .
\end{equation*}%
We note that the twisted duality discovered by F. W. Kamber and Ph. Tondeur
in \cite{KTduality} and the Poincar\'{e} duality for twisted basic
cohomology, proved by the authors in \cite{HabRic2}, are the special cases $%
\theta =0$ and $\theta =\frac{1}{2}\kappa _{b}$, respectively. Using this
duality, we are able to prove in Proposition \ref{mean curv homo Prop} that
a foliated homotopy equivalence between transversely oriented Riemannian
foliations $\left( M,\mathcal{F}\right) $ and $\left( M^{\prime },\mathcal{F}%
^{\prime }\right) $ pulls back the \'{A}lvarez class $\left[ \kappa
_{b}^{\prime }\right] \in H_{d}^{1}\left( M^{\prime },\mathcal{F}^{\prime
}\right) $ to the \'{A}lvarez class $\left[ \kappa _{b}\right] $ $\in
H_{d}^{1}\left( M,\mathcal{F}\right) $. We remark that it has been shown
previously by H. Nozawa in \cite{Noz1}, \cite{Noz2} that the \'Alvarez class
is continuous with respect to smooth deformations of Riemannian foliations.
Finally, in Theorem \ref{BasicSignHomotopyInvtThm}, we show that up to a
sign depending on orientation, the basic signature, now defined on all
foliations admitting a Riemannian foliation structure, is a foliated
homotopy invariant.

\section{Preliminaries\label{preliminariesSection}}

\subsection{Riemannian foliations}

In this section, we will recall some basic facts concerning Riemannian
foliations that could be found in \cite{To}.

Let $(M,\mathcal{F})$ be a closed Riemannian manifold of dimension $n$
endowed with a foliation $\mathcal{F}$ given by an integrable subbundle $%
L\subset TM$ of rank $p$, with $n=p+q$. The subbundle $L=T\mathcal{F}$ is
the tangent bundle to the foliation$.$ Let $Q\cong TM\diagup L$ denote the
normal bundle, and let $g_{Q}$ be a given metric on $Q$. The foliation $(M,%
\mathcal{F},g_{Q})$ is called \textbf{Riemannian} if $\mathcal{L}_{X}g_{Q}=0$
for any section $X\in \Gamma (L)$. In this paper, we will assume that we
have chosen a metric $g$ on $M$ that is bundle-like, meaning through the
isomorphism $Q\cong L^{\bot }$, $g_{Q}=\left. g\right\vert _{L^{\bot }}$.
Such bundle-like metrics always exist. One can show that there exists a
unique metric connection $\nabla $ (with respect to the induced metric) on
the $Q$, called \textbf{transverse Levi-Civita connection}, which is
torsion-free. Recall here that the torsion on $Q$ is being defined as $%
T(X,Y)=\nabla _{X}\pi (Y)-\nabla _{Y}\pi (X)-\pi ([X,Y]),$ where $X$ and $Y$
are vector fields in $\Gamma (TM)$ and $\pi :TM\rightarrow Q$ is the
projection. Such a connection $\nabla $ can be expressed in terms of the
Levi-Civita connection $\nabla ^{M}$ on $M$ as 
\begin{equation*}
\nabla _{X}Y=\left\{ 
\begin{array}{ll}
\pi ([X,Y]), & \text{$\qquad {\forall }X\in \Gamma (L)$ }, \\ 
&  \\ 
\pi (\nabla _{X}^{M}Y), & \text{$\qquad {\forall }X\in \Gamma (Q).$ }%
\end{array}%
\right.
\end{equation*}%
One can also show that the curvature $R^{\nabla }$ associated with the
connection $\nabla $ satisfies $X\lrcorner R^{\nabla }=0$ for all $X\in
\Gamma L$, where the symbol \textquotedblleft $\lrcorner $%
\textquotedblright\ denotes interior product.

\textbf{Basic forms} are differential forms on any foliation $\left( M,%
\mathcal{F}\right) $ that locally depend only on the transverse variables.
That is, they are forms $\alpha \in \Omega (M)$ satisfying the equations $%
X\lrcorner \alpha =X\lrcorner d\alpha =0$ for all $X\in \Gamma (L)$. Let us
denote by $\Omega \left( M,\mathcal{F}\right) \subset \Omega \left( M\right) 
$ the set of all basic forms. In fact, one can easily check that $\Omega
\left( M,\mathcal{F}\right) $ is preserved by the exterior derivative, and
therefore one can associate to $d$ the so-called \textbf{basic cohomology
groups} $H_{d}^{\ast }\left( M,\mathcal{F}\right) $ as 
\begin{equation*}
H_{d}^{k}\left( M,\mathcal{F}\right) =\frac{\ker d_{k}}{\func{image}d_{k-1}}
\end{equation*}%
with%
\begin{equation*}
d_{k}=d:\Omega ^{k}\left( M,\mathcal{F}\right) \rightarrow \Omega
^{k+1}\left( M,\mathcal{F}\right) .
\end{equation*}

\noindent The basic cohomology groups are finite-dimensional for Riemannian
foliations, in which case they satisfy Poincar\'{e} duality if and only if
the foliation is taut. Recall here that a foliation is said to be \textbf{%
taut} if there exists a metric on $M$ so that the mean curvature of the
leaves is zero. Given a bundle-like metric on $\left( M,\mathcal{F}\right) $%
, the mean curvature one-form $\kappa $ is defined by 
\begin{equation*}
\kappa ^{\#}=\sum_{i=1}^{p}\pi \left( \nabla _{f_{i}}^{M}f_{i}\right) ,
\end{equation*}%
where $\left( f_{i}\right) _{1\leq i\leq p}$ is a local orthonormal frame of 
$T\mathcal{F}.$ The orthogonal projection $\kappa _{b}=P\kappa $ of $\kappa $%
, with $P:L^{2}\left( \Omega \left( M\right) \right) \rightarrow L^{2}\left(
\Omega \left( M,\mathcal{F}\right) \right) $, is a closed one-form whose
cohomology class, called the \textbf{\'{A}lvarez class}, in $H_{d}^{1}\left(
M,\mathcal{F}\right) $ is independent of the choice of bundle-like metric
(see \cite{Al}).

Finally, we denote by $\delta _{b}$ the $L^{2}$-adjoint of $d$ restricted to
basic forms (see \cite{To}, \cite{Al}, \cite{PaRi}). Then, for transversely
oriented Riemannian foliations one has 
\begin{equation*}
\delta _{b}=P\delta =\pm \overline{\ast }d\overline{\ast }+\kappa
_{b}\lrcorner =\delta _{T}+\kappa _{b}\lrcorner ,
\end{equation*}%
where $\delta _{T}$ is the formal adjoint of $d$ on the local quotients of
the foliation charts and $\overline{\ast }$ is the pointwise transversal
Hodge star operator defined on all $k$-forms $\gamma $ by%
\begin{equation}
\overline{\ast }\gamma =\left( -1\right) ^{p\left( q-k\right) }\ast \left(
\gamma \wedge \chi _{\mathcal{F}}\right) ,  \label{transvHodgeStarDefn}
\end{equation}%
with $\chi _{\mathcal{F}}$ being the leafwise volume form (or the \textbf{%
characteristic form}) and $\ast $ is the ordinary Hodge star operator.

\subsection{Twisted basic cohomology}

In this section, we shall review some results proved in the paper \cite%
{HabRic2}, where also the definitions of some of the terms below are given.

Given a bundle-like metric on a Riemannian foliation $\left( M,\mathcal{F}%
\right) $ and a basic Clifford bundle $E\rightarrow M$, the basic Dirac
operator is defined as the restriction 
\begin{equation*}
D_{b}=\sum_{i=1}^{q}e_{i}\cdot \nabla _{e_{i}}^{E}-\frac{1}{2}\kappa
_{b}^{\sharp }\cdot ,
\end{equation*}%
to basic sections of $E$, where $\{e_{i}\}_{i=1,\cdots ,q}$ is a local
orthonormal frame of $Q$. The basic Dirac operator preserves the set of
basic sections and is transversally elliptic. From the expression of the
basic Dirac operator applied to the basic Clifford bundle $\Lambda ^{\ast
}Q^{\ast }$, one may write on basic forms 
\begin{eqnarray}
D_{\mathrm{tr}} &=&d+\delta _{T}=d+\delta _{b}-\kappa _{b}\lrcorner :\Omega
^{\mathrm{even}}\left( M,\mathcal{F}\right) \rightarrow \Omega ^{\mathrm{odd}%
}\left( M,\mathcal{F}\right)  \notag \\
D_{b} &=&\frac{1}{2}(D_{\mathrm{tr}}+D_{\mathrm{tr}}^{\ast })s=d-\frac{1}{2}%
\kappa _{b}\wedge +\delta _{b}-\frac{1}{2}\kappa _{b}\lrcorner .
\label{basicDeRhamDefn}
\end{eqnarray}

In \cite{HabRi}, we showed the invariance of the spectrum of $D_{b}$ with
respect to a change of metric on $M$ in any way that leaves the transverse
metric on the normal bundle intact (this includes modifying the subbundle $%
Q\subset TM$, as one must do in order to make the mean curvature basic, for
example). That is,

\begin{theorem}
\label{inv}(In \cite{HabRi}) Let $(M,\mathcal{F})$ be a compact Riemannian
manifold endowed with a Riemannian foliation and basic Clifford bundle $%
E\rightarrow M$. The spectrum of the basic Dirac operator is the same for
every possible choice of bundle-like metric that is associated to the
transverse metric on the quotient bundle $Q$.
\end{theorem}

In \cite{HabRic2}, the authors define the new cohomology $\widetilde{H}%
^{\ast }\left( M,\mathcal{F}\right) =H_{d-\frac{1}{2}\kappa _{b}\wedge
}^{\ast }\left( M,\mathcal{F}\right) $ (called the \textbf{twisted basic
cohomology}) of basic forms, using $\widetilde{d}:=d-\frac{1}{2}\kappa
_{b}\wedge $ as a differential. Recall from (\ref{basicDeRhamDefn}) that the
basic de Rham operator is $D_{b}=\widetilde{d}+\widetilde{\delta },$ where $%
\widetilde{\delta }:=\delta _{b}-\frac{1}{2}\kappa _{b}\lrcorner .$ Because $%
\kappa _{b}$ is basic and closed, the twisted differential preserves $\Omega
\left( M,\mathcal{F}\right) $, $\widetilde{d}^{2}=0$ and $\widetilde{\delta }%
^{2}=0$. We show that the corresponding Betti numbers and eigenvalues of the
twisted basic Laplacian $\widetilde{\Delta }:=\widetilde{d}\,\widetilde{%
\delta }+\widetilde{\delta }\widetilde{d}$ are independent of the choice of
a bundle-like metric. In the remainder of this section, we assume that the
foliation is transversely oriented so that the $\overline{\ast }$ operator
is well-defined.

\begin{definition}
We define the basic $\widetilde{d}$-cohomology $\widetilde{H}^{\ast }\left(
M,\mathcal{F}\right) $ by 
\begin{equation*}
\widetilde{H}^{k}\left( M,\mathcal{F}\right) =\frac{\ker \widetilde{d}%
_{k}:\Omega ^{k}\left( M,\mathcal{F}\right) \rightarrow \Omega ^{k+1}\left(
M,\mathcal{F}\right) }{\mathrm{image}~\widetilde{d}_{k-1}:\Omega
^{k-1}\left( M,\mathcal{F}\right) \rightarrow \Omega ^{k}\left( M,\mathcal{F}%
\right) }.
\end{equation*}
\end{definition}

\begin{proposition}
\label{indepTwBasicCohomCor}(in \cite{HabRic2}) The dimensions of $%
\widetilde{H}^{k}\left( M,\mathcal{F}\right) $ are independent of the choice
of the bundle-like metric and independent of the transverse Riemannian
foliation structure.
\end{proposition}

\subsection{The basic signature operator}

We suppose that $\left( M,\mathcal{F},g_{Q}\right) $ is a transversally
oriented Riemannian foliation of even codimension $q$, and let $g_{M}$ be a
bundle-like metric. Let 
\begin{equation*}
\bigstar =i^{k\left( k-1\right) +\frac{q}{2}}\overline{\ast }
\end{equation*}%
as an operator on basic $k$-forms, analogous to the involution used to
identify self-dual and anti-self-dual forms on a manifold. Note that this
endomorphism is symmetric, and 
\begin{equation*}
\bigstar ^{2}=1.
\end{equation*}%
In the particular case when $q=4m$ for an integer $m$, we have $%
\bigstar =\overline{\ast }$ on $2m$-forms.

\begin{proposition}
(In \cite{HabRic2}) We have \vspace{0in}$\bigstar \left( \widetilde{d}+%
\widetilde{\delta }\right) =-\left( \widetilde{d}+\widetilde{\delta }\right)
\bigstar $. In fact, $\bigstar \widetilde{d}=-\widetilde{\delta }\bigstar $
and $\bigstar \widetilde{\delta }=-\widetilde{d}\bigstar $.
\end{proposition}

Let $\Omega^{+}\left( M,\mathcal{F}\right) $ denote the $+1$ eigenspace of $%
\bigstar $ in $\Omega^{\ast }\left( M,\mathcal{F}\right) $, and let $%
\Omega^{-}\left( M,\mathcal{F}\right) $ denote the $-1$ eigenspace of $%
\bigstar $ in $\Omega^{\ast }\left( M,\mathcal{F}\right) $. By the
proposition above, $D_{b}=\widetilde{d}+\widetilde{\delta }$ maps $%
\Omega^{\pm }\left( M,\mathcal{F}\right) $ to $\Omega^{\mp }\left( M,%
\mathcal{F}\right) $. Therefore, we may define the basic signature operator
as follows.

\begin{definition}
On a transversally oriented Riemannian foliation of even codimension, let
the \textbf{basic signature operator} be the operator $D_{b}:\Omega
^{+}\left( M,\mathcal{F}\right) \rightarrow \Omega ^{-}\left( M,\mathcal{F}%
\right) $. We define the \textbf{basic signature }$\sigma \left( M,\mathcal{F%
}\right) $ \textbf{of the foliation} to be the index%
\begin{equation*}
\sigma \left( M,\mathcal{F}\right) =\dim \ker \left( \left. \widetilde{%
\Delta }\right\vert _{\Omega ^{+}\left( M,\mathcal{F}\right) }\right) -\dim
\ker \left( \left. \widetilde{\Delta }\right\vert _{\Omega ^{-}\left( M,%
\mathcal{F}\right) }\right).
\end{equation*}
\end{definition}

\begin{remark}
We note that such a definition is not possible for the operator $d+\delta
_{b}$, because the relationship in the proposition above does not hold for $%
d+\delta _{b}$.
\end{remark}

\subsection{Known results on the homotopy invariance of the basic cohomology
and signature in the taut case}

\noindent In this section, we review the results in \cite{BenRA}, where the
smooth homotopy invariance of ordinary basic cohomology is proved and the
basic signature is studied in the case where the foliation is taut. For
this, given two foliated manifolds $(M,\mathcal{F})$ and $(M^{\prime },%
\mathcal{F^{\prime }})$, we say that a map $f:(M,\mathcal{F})\rightarrow
(M^{\prime },\mathcal{F^{\prime }})$ is \textbf{foliated} if $f$ maps the
leaves of $\mathcal{F}$ to the leaves of $\mathcal{F^{\prime }}$, i.e. $%
f_{\ast }(T\mathcal{F})\subset T\mathcal{F}^{\prime }$. The following fact
is well-known and easy to show.

\begin{lemma}
\label{linearmapLemma} If $f:(M,\mathcal{F})\rightarrow (M^{\prime },%
\mathcal{F^{\prime }})$ is foliated, then $f^*(\Omega(M^\prime,\mathcal{F}%
^\prime)) \subseteq\Omega(M,\mathcal{F})$.
\end{lemma}

\begin{definition}
\label{def:homotopic} Let $(M,\mathcal{F})$ and $(M^{\prime },\mathcal{%
F^{\prime }})$ be two foliated manifolds. We say that the foliated maps $%
\phi :M\rightarrow M^{\prime }$ and $\psi :M\rightarrow M^{\prime }$ are 
\textbf{foliated homotopic} if there exists a continuous map $H:[0,1]\times
M\rightarrow M^{\prime }$ such that $H(0,x)=\phi (x),\,H(1,x)=\psi (x)$ and
that for each $t\in \lbrack 0,1]$ the map $H(t,\cdot )$ is smooth and
foliated.
\end{definition}

\begin{definition}
We say that a foliated map $f:\left( M,\mathcal{F}\right) \rightarrow \left(
M^{\prime },\mathcal{F}^{\prime }\right) $ is a \textbf{foliated homotopy
equivalence} if there exists a foliated map $g:M^{\prime }\rightarrow M$
such that $f\circ g$ is foliated homotopic to $\mathrm{Id}_{M^{\prime }}$
and $g\circ f$ is foliated homotopic to $\mathrm{Id}_{M}.$
\end{definition}

\begin{proposition}
(In \cite{BenRA}; also in \cite{EKN} for the case of foliated
homeomorphisms) If a map $f:M\rightarrow M^{\prime }$ is a smooth foliated
homotopy equivalence, then $f^{\ast }$ induces an isomorphism between $%
H_{d}^{\ast }\left( M^{\prime },\mathcal{F}^{\prime }\right) $ and $%
H_{d}^{\ast }\left( M,\mathcal{F}\right) .$
\end{proposition}

In what follows, for a taut Riemannian foliation $\left( M,\mathcal{F}%
\right) $ of codimension $2\ell $, we let $A_{\mathcal{F}}:H_{d}^{\ell }(M,%
\mathcal{F})\times H_{d}^{\ell }(M,\mathcal{F})\rightarrow \mathbb{R}$ be
the bilinear form%
\begin{equation*}
A_{\mathcal{F}}\left( \left[ \alpha \right] ,\left[ \beta \right] \right)
=\int_{M}\alpha \wedge \beta \wedge \chi _{\mathcal{F}},
\end{equation*}%
which can be seen to be well-defined (for the taut case only). If $\ell$ is even,
it is easy to see that 
\begin{eqnarray*}
\sigma \left( M,\mathcal{F}\right) &=&\dim \ker \left( \left. \Delta
\right\vert _{\Omega ^{+}\left( M,\mathcal{F}\right) }\right) -\dim \ker
\left( \left. \Delta \right\vert _{\Omega ^{-}\left( M,\mathcal{F}\right)
}\right) \\
&=&\text{signature}\left( A_{\mathcal{F}}(\cdot,\cdot)\right),
\end{eqnarray*}%
because $\bigstar d=-\delta \bigstar $ and $\bigstar \delta =-d\bigstar $
and $\bigstar =\overline{\ast }$ on $\ell$-forms. If $\ell $ is odd, it can be seen
easily that $A_{\mathcal{F}}([\alpha],[\alpha])\equiv 0$ for all $[\alpha]\in
H_d^\ell (M,\mathcal{F})$, and also that $\sigma(M,\mathcal{F})=0$
since $\bigstar=i \overline{\ast}$ on $\ell$-forms so that the kernels have the same
complex dimension.
\begin{theorem}
(In \cite{BenRA}) Let $f:M\rightarrow M^{\prime }$ be a smooth foliated
homotopy equivalence between two taut Riemannian foliations of codimension $%
2\ell $ and transverse volume forms $\nu $, $\nu ^{\prime }$, respectively.
Then $\sigma (M,\mathcal{F})=\sigma (M^{\prime },\mathcal{F^{\prime }})$ if $%
f$ preserves the transverse orientation and $\sigma (M,\mathcal{F})=-\sigma
(M^{\prime },\mathcal{F^{\prime }})$ otherwise.
\end{theorem}

\section{Homotopy invariance of twisted cohomology and basic signature}

\subsection{Twisted basic cohomology and basic signature for general
Riemannian foliations\label{twistedBscCohomBscSignSubsection}}

Let $q$ be the codimension of the transversally oriented foliation $\mathcal{%
F}$ in $M$, and let $\overline{\ast }$ denote the transversal Hodge star
operator. This operator is defined by \eqref{transvHodgeStarDefn} but can
then be extended to a map on cohomology classes. For example, from \cite%
{HabRic2} we have formulas such as 
\begin{align*}
\delta _{b}\overline{\ast }& =(-1)^{k+1}\overline{\ast }(d-\kappa _{b}\wedge
), \\
d\overline{\ast }& =(-1)^{k}\overline{\ast }(\delta _{b}-\kappa
_{b}\lrcorner ),
\end{align*}%
so that $\overline{\ast }$ maps $(d-\kappa _{b}\wedge ,\delta _{b}-\kappa
_{b}\lrcorner )$-harmonic forms to $(d,\delta _{b})$-harmonic forms. Using
the corresponding Hodge theorems, this gives a map (actually an isomorphism)
between the corresponding cohomology groups. That is, we can use $\overline{%
\ast }$ to define the linear map 
\begin{equation*}
\overline{\ast }:H_{d-\kappa _{b}\wedge }^{k}(M,\mathcal{F})\overset{\cong }{%
\longrightarrow }H_{d}^{q-k}(M,\mathcal{F}).
\end{equation*}%
This was originally observed in \cite{KTduality} for the case of basic mean
curvature and can be adjusted using the techniques in \cite{Al} and \cite%
{PaRi} for the general case.

\begin{theorem}
\label{PairingThm}Let $(M,\mathcal{F})$ be a Riemannian foliation of
codimension $q$ that is transversally oriented. There for integers $0\leq
r,s\leq q$, there is a pairing 
\begin{equation*}
\widetilde{H}^{r}(M,\mathcal{F})\times \widetilde{H}^{s}(M,\mathcal{F}%
)\rightarrow H_{d}^{q-r-s}(M,\mathcal{F}),
\end{equation*}%
defined as follows. For $[\alpha ]\in \widetilde{H}^{r}(M,\mathcal{F})$ and $%
[\beta ]\in \widetilde{H}^{s}(M,\mathcal{F})$, $[\alpha \wedge \beta ]$
defines a class in $H_{d-\kappa _{b}\wedge }^{r+s}(M,\mathcal{F})$, and thus 
$\overline{\ast }[\alpha \wedge \beta ]$ is a class in ordinary basic
cohomology $H_{d}^{q-r-s}(M,\mathcal{F})$. In the particular case when $r=s=%
\frac{q}{2}$, this pairing is nondegenerate, and the result is 
\begin{equation*}
\overline{\ast }[\alpha \wedge \beta ]=\left[ \int_{M}\alpha \wedge \beta
\wedge \chi \right] \in H_{d}^{0}(M,\mathcal{F})\cong \mathbb{R}.
\end{equation*}
\end{theorem}

\begin{proof}
We have 
\begin{eqnarray*}
d(\alpha \wedge \beta ) &=&d(\alpha )\wedge \beta +(-1)^{r}\alpha \wedge
d(\beta ) \\
&=&\frac{1}{2}\kappa _{b}\wedge \alpha \wedge \beta +\frac{(-1)^{r}}{2}%
\alpha \wedge \kappa _{b}\wedge \beta  \\
&=&\kappa _{b}\wedge \alpha \wedge \beta .
\end{eqnarray*}%
Then $\alpha \wedge \beta $ defines a cohomology class in $H_{d-\kappa
_{b}\wedge }^{r+s}(M,\mathcal{F})$. We then apply $\overline{\ast }$ to the
associated basic harmonic form to get an element in $H_{d}^{q-r-s}(M,%
\mathcal{F})$. Note that the class $\left[ \alpha \wedge \beta \right] \in
H_{d-\kappa _{b}\wedge }^{r+s}(M,\mathcal{F})$ is well-defined. If $\alpha
^{\prime }=\alpha +\widetilde{d}\gamma $ with $\gamma \in \Omega
^{r-1}\left( M,\mathcal{F}\right) $, then 
\begin{eqnarray*}
\widetilde{d}\gamma \wedge \beta  &=&d\gamma \wedge \beta -\frac{1}{2}\kappa
_{b}\wedge \gamma \wedge \beta  \\
&=&d\left( \gamma \wedge \beta \right) -\left( -1\right) ^{r-1}\gamma \wedge
d\beta +\frac{1}{2}\left( -1\right) ^{r-1}\gamma \wedge \kappa _{b}\wedge
\beta -\kappa _{b}\wedge \gamma \wedge \beta  \\
&=&\left( d-\kappa _{b}\wedge \right) \left( \gamma \wedge \beta \right)
-\left( -1\right) ^{r-1}\gamma \wedge \left( d-\frac{1}{2}\kappa _{b}\wedge
\right) \beta  \\
&=&\left( d-\kappa _{b}\wedge \right) \left( \gamma \wedge \beta \right) .
\end{eqnarray*}%
It follows that the result $\left[ \alpha \wedge \beta \right] $ is
independent of the representative of the class $\left[ \alpha \right] $. By
a similar argument, it is independent of the choice of $\beta $. It follows
that $\overline{\ast }[\alpha \wedge \beta ]\in H_{d}^{q-r-s}(M,\mathcal{F})$
is well-defined. \newline
When $r=s=\frac{q}{2}$, we note that for any nonzero class $[\alpha ]\in 
\widetilde{H}^{r}(M,\mathcal{F})$, $[\overline{\ast }\alpha ]\in \widetilde{H%
}^{r}(M,\mathcal{F})$ by \cite{HabRic2}, and so $\alpha \wedge \overline{%
\ast }\alpha $ is a multiple of the transverse volume form, so 
\begin{equation*}
\overline{\ast }[\alpha \wedge \overline{\ast }\alpha ]=\langle \alpha
,\alpha \rangle \neq 0.
\end{equation*}%
Repeating the argument in the second slot shows that the pairing is
nondegenerate.
\end{proof}

Suppose that $\left( M,\mathcal{F}\right) $ is a Riemannian foliation of
codimension $2\ell $, and we define the bilinear map $A_{\mathcal{F}}:\Omega
^{\ell }(M,\mathcal{F})\times \Omega ^{\ell }(M,\mathcal{F})\rightarrow 
\mathbb{R}$ by%
\begin{equation*}
A_{\mathcal{F}}\left( \alpha ,\beta \right) =\int_{M}\alpha \wedge \beta
\wedge \chi _{\mathcal{F}}.
\end{equation*}

\begin{proposition}
\vspace{0in}The induced map $A_{\mathcal{F}}:\widetilde{H}^{\ell }(M,%
\mathcal{F})\times \widetilde{H}^{\ell }(M,\mathcal{F})\rightarrow \mathbb{R}
$ is well-defined.
\end{proposition}

\begin{proof}
This is a direct consequence of Theorem \ref{PairingThm}.
\end{proof}

\begin{lemma}
\label{SignPairingLemma}The basic signature $\sigma \left( M,\mathcal{F}%
\right) $ of a Riemannian foliation of codimension $2l$ is the same as the
signature of the quadratic form $A_{\mathcal{F}}([\alpha],[\alpha])$ for $\alpha\in\widetilde{H}^{\ell }(M,%
\mathcal{F})$.
\end{lemma}

\begin{proof}
When $\ell$ is even,
$\bigstar \widetilde{d}=-\widetilde{\delta }\bigstar $ and $\bigstar 
\widetilde{\delta }=-\widetilde{d}\bigstar$, so we compute for any $\widetilde{\Delta}$-harmonic $%
\ell$-form $\alpha=\bigstar \alpha = \overline{\ast}\alpha$ in $%
\Omega^{+}\left( M,\mathcal{F}\right)$, 
\begin{eqnarray*}
A_{\mathcal{F}}(\alpha ,\alpha ) &=&\int_{M}\alpha \wedge \alpha \wedge \chi
_{\mathcal{F}} \\
&=&\int_{M}\alpha \wedge \overline{\ast }\alpha \wedge \chi _{\mathcal{F}%
}=\int_M |\alpha |^{2}.
\end{eqnarray*}%
In the same way, we find $A_{\mathcal{F}}(\beta ,\beta )=-\displaystyle\int_M |\beta |^{2}$ for
any harmonic $\ell $-form $\beta $ in $\Omega ^{-}\left( M,\mathcal{F}%
\right).$ Therefore, 

\begin{eqnarray*}
\sigma \left( M,\mathcal{F}\right) &=&\dim \ker \left( \left. \widetilde{%
\Delta }\right\vert _{\Omega ^{+}\left( M,\mathcal{F}\right) }\right) -\dim
\ker \left( \left. \widetilde{\Delta }\right\vert _{\Omega ^{-}\left( M,%
\mathcal{F}\right) }\right) \\
&=&\text{signature}\left( A_{\mathcal{F}}(\cdot,\cdot)\right) .
\end{eqnarray*}%
If $\ell $ is odd, it can be seen easily that $A_{\mathcal{F}}([\alpha],[\alpha])\equiv 0$ for all $[\alpha]$,
and since again $\bigstar=i\overline{\ast}$ on $\ell$-forms in this case, the kernels have the same dimension 
so that $\sigma(M,\mathcal{F})=0$.
\end{proof}

We will need the following lemma; this is known to experts but does not appear to be present in the literature.
\begin{lemma}
\label{bundlelikeMetricsHomotopic}
If $(M,\mathcal{F})$ is a smooth foliation on a (not necessarily closed) manifold that admits a bundle-like metric,
then any two such bundle-like metrics are homotopic through a smooth family of bundle-like metrics.
\end{lemma}
\begin{proof}
Consider a smooth foliation $(M,\mathcal{F})$ of codimension $q$ and dimension $p$ on which a bundle-like metric is defined.
Near any point we may choose a foliation chart with adapted coordinates $(x,y)\in \mathbb{R}^p\times\mathbb{R}^q$ on
which an adapted local orthonormal frame $(b_1,b_2,...,b_p,e_1,...,e_q)$ is defined. Let  
$(b^1,b^2,...,b^p,e^1,...,e^q)$ be the corresponding coframe. Then the bundle-like metric takes the form
\[
g=ds^2=\sum_{j=1}^p (b^j)^2 + \sum_{k=1}^q (e^k)^2,
\]
with 
\[
g_Q=\sum_{k=1}^q (e^k)^2=\sum_{\alpha,\beta=1}^q h_{\alpha\beta} dy^\alpha dy^\beta
\] 
for a positive-definite symmetric matrix of functions $(h_{\alpha\beta})$ and 
\[
g_L=\sum_{j=1}^p (b^j)^2
\]
positive definite when restricted to $L=T\mathcal{F}$. 
The bundle-like condition is equivalent to $h_{\alpha\beta}$ being a matrix of basic functions; that is, its restriction to $Q=T\mathcal{F}^\perp$ must be holonomy-invariant;
see \cite[Section IV, Proposition 4.2]{Re1}. Now, suppose that  $\widetilde{g}$ is another such bundle-like metric; therefore, in a possibly smaller foliation chart we have
\[
\widetilde{g}=\sum_{j=1}^p (\widetilde{b}^j)^2 + \sum_{\alpha,\beta=1}^q \widetilde{h}_{\alpha\beta}(y) dy^\alpha dy^\beta=\widetilde{g}_L+\widetilde{g}_Q,
\]
noting that the normal bundle $\widetilde{Q}$ for  $\widetilde{g}$ is typically different from that of $g$.
Let $\Pi:TM\to L$ be the orthogonal projection defined by the first metric $g$. Since the tangential part of $\widetilde{g}$ also remains positive definite on $L$,
$(\Pi^*\widetilde{b}^1,...,\Pi^*\widetilde{b}^p,\widetilde{e}^1,...,\widetilde{e}^q)$ forms a basis of $T^*M$, and so choosing them to be an orthonormal basis defines a new bundle-like
metric
\[
\overline{g}=\sum_{j=1}^p (\Pi^*\widetilde{b}^j)^2 + \sum_{\alpha,\beta=1}^q \widetilde{h}_{\alpha\beta}(y) dy^\alpha dy^\beta=(\Pi^*\otimes\Pi^*)(\widetilde{g}_L)+\widetilde{g}_Q,
\]
with the feature that the bundles $L$ and $Q$ agree (and thus $L^*$ and $Q^*$ agree) for both $\overline{g}$ and $g$.
It is clear that $\widetilde{g}$ and $\overline{g}$ are homotopic through a homotopy transforming $\widetilde{b}^j$ to $\Pi^*\widetilde{b}^j$; specifically, for $0\le t \le 1 $ 
we may set  $b^j(t)=(1-t)\widetilde{b}^j+t\Pi^*\widetilde{b}^j$,
and then the resulting metric homotopy is
\[
g_t=\sum (b^j(t))^2+\widetilde{g}_Q;  \quad g_0=\widetilde{g},\, g_1=\overline{g}.
\]
The homotopy is independent of the choice of coframe $\{ \widetilde{b}^{j}\}$, because if $U=(U_{\ell m})$ is any orthogonal matrix of functions and $\widetilde{b}^{j\prime}=\sum_m U_{jm}\widetilde{b}^m$, then 
\begin{eqnarray*}
\sum_j(b^{j\prime}(t))^2&=&\sum_j(\sum_m(1-t)U_{jm}\widetilde{b}^m+tU_{jm}\Pi^*\widetilde{b}^m)^2 \\
&=&\sum_j(\sum_mU_{jm}((1-t)\widetilde{b}^m+t\Pi^*\widetilde{b}^m))^2 \\
&=&\sum_j\sum_{\ell,m}U_{j\ell}((1-t)\widetilde{b}^\ell+t\Pi^*\widetilde{b}^\ell)U_{jm}((1-t)\widetilde{b}^m+t\Pi^*\widetilde{b}^m) \\
&=&\sum_\ell ((1-t)\widetilde{b}^\ell+t\Pi^*\widetilde{b}^\ell)^2 = \sum_\ell (b^{\ell}(t))^2.
\end{eqnarray*}
Thus, this homotopy is independent of coordinates and choice of frame. Next, $\overline{g}$ and $g$ are homotopic through a convex combination of the respective metrics on  $L$ and $Q$; specifically, letting
$\overline{g} = \overline{g}_L+\overline{g}_Q$, for  $t\in [0,1]$, we have
\[
h_t = (1-t) \overline{g}_L+t g_L + (1-t) \overline{g}_Q + t g_Q
\]
is a family of metrics that satisfies the bundle-like condition for each $t$.
We may now form the following smooth homotopy between $\widetilde{g}$ and $g$:
\[
G_t = \begin{cases} g_{u(t)} \quad & \text{for }0\le t \le \frac 12 ,\\
h_{v(t)} & \text{for }\frac 12 \le t \le 1,
\end{cases}
\]
where $u,v:\mathbb{R}\to \mathbb{R}$ are smooth increasing functions such that $u\equiv 0$ on $(-\infty,0]$, $u\equiv 1$ on $[\frac 12,\infty)$ and $v\equiv 0$ on $(-\infty,\frac 12)$ and $v\equiv 1$ on $[1,\infty)$. 
\end{proof}

\begin{proposition}
\label{BscSignIsInvtProp}The basic signature $\sigma \left( M,\mathcal{F}%
\right) $ of a Riemannian foliation does not depend on the transverse
Riemannian structure or the bundle-like metric; it is a smooth invariant of
the foliation.
\end{proposition}

\begin{proof}
Observe that by the previous lemma, any two bundle-like metrics on $(M,\mathcal{F})$ are smoothly homotopic through 
bundle-like metrics, and it follows that the principal transverse symbols of the 
signature operators $\widetilde{d}+\widetilde{\delta}$ on $\Omega^\pm (M,\mathcal{F})$ 
with respect to those metrics
are smoothly homotopic. Since the basic signature is the index of this operator on basic sections, there
is a continuous path through Fredholm operators connecting the two operators, so that the index cannot 
change along that path. Thus, the basic signature is a smooth invariant of the foliation.
See \cite{EK} and \cite{BKRfol} for properties of the basic index.

Note that it is also possible to see this result through a long, detailed analysis of the differentials 
and bundle-like metrics and the effects on $\chi_{\mathcal{F}}$ and $\kappa_b$ as
in \cite{Al}.
\end{proof}


\subsection{Basic Lichnerowicz cohomology and foliated homotopy invariance 
\label{BscLichnCohomHomotpyInvSubsection}}

\noindent We start with any smooth foliation $\left( M,\mathcal{F}\right) $.
In what follows, let $\theta $ be a closed basic one-form. Then $d+\theta
\wedge $ is a differential on the space of basic forms. Let $H_{d+\theta
\wedge }^{\ast }\left( M,\mathcal{F}\right) $ denote the resulting
cohomology, which is sometimes called basic Lichnerowicz cohomology or basic
Morse-Novikov cohomology; see \cite{Va}, \cite{Ban}, \cite{IdaP}, \cite{AitH}%
, \cite{OrnSle}.

\begin{lemma}
(\cite[Proposition 3.0.11]{AitH}) \label{isomorphicLichCohomLem}If $\left[
\alpha \right] =\left[ \beta \right] \in H_{d}^{1}\left( M,\mathcal{F}%
\right) $, then $H_{d+\alpha \wedge }^{\ast }\left( M,\mathcal{F}\right)
\cong H_{d+\beta \wedge }^{\ast }\left( M,\mathcal{F}\right) $.
\end{lemma}

\begin{lemma}
\label{LichnLemma}Let $f:(M,\mathcal{F})\rightarrow (M^{\prime },\mathcal{F}%
^{\prime })$ be a foliated map, and let $\theta $ be a closed basic one
form. Then $f^{\ast }:\Omega \left( M^{\prime },\mathcal{F}^{\prime }\right)
\rightarrow \Omega \left( M,\mathcal{F}\right) $ induces a linear map from $%
H_{d+\theta \wedge }^{\ast }\left( M^{\prime },\mathcal{F}^{\prime }\right) $
to $H_{d+f^{\ast }\theta \wedge }^{\ast }\left( M,\mathcal{F}\right) $.
\end{lemma}

\begin{proof}
By Lemma \ref{linearmapLemma}, $f^{\ast }\Omega \left( M^{\prime },\mathcal{F%
}^{\prime }\right) \subseteq \Omega \left( M,\mathcal{F}\right)$. We must
prove that the linear map $f^{\ast }$ maps closed and exact forms to closed
and exact forms, respectively. Let $[\alpha ]\in H_{d+\theta \wedge
}^{k}\left( M^{\prime },\mathcal{F}^{\prime }\right) $, then 
\begin{eqnarray*}
\left( d+f^{\ast }\theta \wedge \right) (f^{\ast }\alpha ) &=&d(f^{\ast
}\alpha )+f^{\ast }\theta \wedge f^{\ast }\alpha =f^{\ast }(d\alpha
)+f^{\ast }\left( \theta \wedge \alpha \right) \\
&=&f^{\ast }\left( (d+\theta \wedge )\alpha \right) =0.
\end{eqnarray*}%
Thus closed forms on $M^{\prime }$ are mapped to closed forms on $M$. Next,
for any $\beta \in \Omega ^{k-1}\left( M^{\prime },\mathcal{F}^{\prime
}\right) $, 
\begin{eqnarray*}
f^{\ast }\left( (d+\theta \wedge )\beta \right) &=&d\left( f^{\ast }\beta
\right) +f^{\ast }\theta \wedge f^{\ast }\beta \\
&=&\left( d+f^{\ast }\theta \wedge \right) f^{\ast }\beta ,
\end{eqnarray*}%
so that exact forms map to exact forms.
\end{proof}

Let us consider two manifolds $(M,\mathcal{F})$ and $(M^{\prime },\mathcal{%
F^{\prime }})$ endowed with a Riemannian foliations $\mathcal{F}$ and $%
\mathcal{F^{\prime }}.$ We denote by $\kappa $ (resp. $\kappa ^{\prime }$)
the mean curvature of the foliation $\mathcal{F}$ (resp. $\mathcal{F^{\prime
}}$), with metrics chosen so that both mean curvatures forms are basic. By 
\cite{Dom}, this can always be done.

\begin{proposition}
Let $f:(M,\mathcal{F})\rightarrow (M^{\prime },\mathcal{F^{\prime }})$ be a
foliated map. Suppose that a bundle-like metric $g_{M^{\prime }}$ on $%
M^{\prime }$ is given such that the mean curvature $\kappa ^{\prime }$ is
basic. Suppose that the basic cohomology class $[f^{\ast }(\kappa ^{\prime
})]\in H_{d}^{1}\left( M,\mathcal{F}\right) \ $contains the mean curvature
one-form for some bundle-like metric on $M$. Then $f^{\ast }$ induces a
linear map from $\widetilde{H}^{\ast }(M^{\prime },g_{M^{\prime }})$ to $%
\widetilde{H}^{\ast }\left( M,g_{M}\right) $ with respect to some
bundle-like metric $g_{M}$ on $M$ such that $\kappa =f^{\ast }(\kappa
^{\prime })$.
\end{proposition}

\begin{proof}
By Lemma \ref{linearmapLemma}, $f^{\ast }\Omega \left( M^{\prime },\mathcal{F%
}^{\prime }\right) \subseteq \Omega \left( M,\mathcal{F}\right) $. We are
given that for some given $\widetilde{g}_{M}$, $\widetilde{\kappa }=f^{\ast
}\kappa ^{\prime }+d\eta $ for some basic function $\eta $. We then choose
multiply the metric in the leaf direction by $\exp \left( \eta \right) $ and
obtain a bundle-like metric $g_{M}$ such that $\kappa =f^{\ast }\kappa
^{\prime }$. Lemma \ref{LichnLemma} completes the proof with $\theta =-\frac{%
1}{2}\kappa ^{\prime }$.
\end{proof}

\begin{lemma}
Let $H:I\times M\rightarrow M^{\prime }$ be a smooth foliated homotopy from $%
\left( M,\mathcal{F}\right) $ to $\left( M^{\prime },\mathcal{F}^{\prime
}\right) $, and let $\theta ^{\prime }$ be a closed basic one-form on $%
M^{\prime }$. Let 
\begin{equation*}
a_{t}=\exp \left( \int_{0}^{t}j_{s}^{\ast }\left( \partial _{t}\lrcorner
H^{\ast }\theta ^{\prime }\right) ds\right) \in \Omega ^{0}\left( M,\mathcal{%
F}\right) ,
\end{equation*}%
where $0\leq t\leq 1$ and $j_{s}:M\rightarrow I\times M$ is defined by $%
j_{s}\left( \cdot \right) =\left( s,\cdot \right) $. Let $h:\Omega
^{k}\left( M^{\prime },\mathcal{F}^{\prime }\right) \rightarrow \Omega
^{k-1}\left( M,\mathcal{F}\right) $ be defined by 
\begin{equation*}
h\left( \sigma \right) =\int_{0}^{1}a_{s}~j_{s}^{\ast }\left( \partial
_{t}\lrcorner H^{\ast }\sigma \right) ~ds.
\end{equation*}%
Then%
\begin{equation*}
a_{1}\left( \cdot \right) H\left( 1,\cdot \right) ^{\ast }-H\left( 0,\cdot
\right) ^{\ast }=\left( d+H\left( 0,\cdot \right) ^{\ast }\theta ^{\prime
}\wedge \right) h+h\left( d+\theta ^{\prime }\wedge \right)
\end{equation*}%
as operators on $\Omega ^{\ast }\left( M^{\prime },\mathcal{F}^{\prime
}\right) .$
\end{lemma}

\begin{proof}
The proof is exactly the same as the proof of Lemma 1.1 in \cite{HalRyb},
but with basic functions and forms. With the definitions given, we just use
the fact that $H\left( t,\cdot \right) ^{\ast }\theta ^{\prime }-H\left(
0,\cdot \right) ^{\ast }\theta ^{\prime }=d\left( \log \left\vert
a_{t}\right\vert \right) $ and calculate the derivatives.
\end{proof}

Using the chain homotopy $h$ in the Lemma above, we get the following with $%
H\left( 0,\cdot \right) =\phi $, $H\left( 1,\cdot \right) =\psi $, $%
a_{1}\left( \cdot \right) =f$. Also, because of \cite[Corollary 13.3]{ALMa},
if two smooth foliated maps are (continuously) homotopic, then there exists
a smooth homotopy between them.

\begin{corollary}
(Homotopy invariance of basic Lichnerowicz cohomology)\label%
{HomotopyInvCorollary} Let $\phi $ and $\psi $ be two smooth maps that are
foliated homotopic from $\left( M,\mathcal{F}\right) $ to $\left( M^{\prime
},\mathcal{F}^{\prime }\right)$, and let $\theta ^{\prime }$ be a closed
basic one-form on $M^{\prime }$. Then there exists a positive basic function 
$\mu $ on $M$ such that $\phi ^{\ast }=\mu \psi ^{\ast }:H_{d+\theta
^{\prime }\wedge }^{\ast }\left( M^{\prime },\mathcal{F}^{\prime }\right)
\rightarrow H_{d+\left( \phi ^{\ast }\theta ^{\prime }\right) \wedge }^{\ast
}\left( M,\mathcal{F}\right) $.
\end{corollary}

\begin{proposition}
(Also in \cite{EKN} for the case of foliated homeomorphisms and $\theta =0$) %
\label{FolHomotpEquivIsoProp}If a map $f:\left( M,\mathcal{F}\right)
\rightarrow \left( M^{\prime },\mathcal{F}^{\prime }\right) $ is a foliated
homotopy equivalence and $\theta ^{\prime }$ is a closed basic one-form on $%
M^{\prime }$, then $f^{\ast }$ induces an isomorphism between $H_{d+\theta
^{\prime }\wedge }^{\ast }(M^{\prime },\mathcal{F}^{\prime })$ and $%
H_{d+f^{\ast }\theta ^{\prime }\wedge }^{\ast }(M,\mathcal{F}).$
\end{proposition}

\begin{proof}
Given $f$ and $g$ as in the definition, by Lemma \ref{LichnLemma}, we have
linear maps 
\begin{equation*}
H_{d+\theta \wedge }^{\ast }(M^{\prime },\mathcal{F}^{\prime })\overset{%
f^{\ast }}{\rightarrow }H_{d+f^{\ast }\theta \wedge }^{\ast }(M,\mathcal{F})%
\overset{g^{\ast }}{\rightarrow }H_{d+g^{\ast }f^{\ast }\theta \wedge
}^{\ast }(M^{\prime },\mathcal{F}^{\prime }).
\end{equation*}%
Since $f\circ g$ is foliated homotopic to the identity, by Corollary \ref%
{HomotopyInvCorollary}, there exists a positive basic function $\mu $ such
that 
\begin{equation*}
\text{id}=\mu g^{\ast }f^{\ast }=\mu \left( f\circ g\right) ^{\ast
}:H_{d+\theta ^{\prime }\wedge }^{\ast }\left( M^{\prime },\mathcal{F}%
^{\prime }\right) \rightarrow H_{d+\theta ^{\prime }\wedge }^{\ast }\left(
M^{\prime },\mathcal{F}^{\prime }\right) .
\end{equation*}%
In particular, $g^{\ast }f^{\ast }\theta ^{\prime }=\theta ^{\prime }+\frac{%
d\mu }{\mu }$. After considering the map $g\circ f$, we then see that we
must have that $f^{\ast }$ and $g^{\ast }$ are isomorphisms, since
multiplication by $\mu $ is also an isomorphism on basic Lichnerowicz
cohomology.
\end{proof}

\subsection{Invariance of the basic signature}

\label{InvBscSignSubsection}

\begin{proposition}
\label{codimensionProp}If $f:\left( M,\mathcal{F}\right) \rightarrow \left(
M^{\prime },\mathcal{F}^{\prime }\right) $ is a foliated homotopy
equivalence between Riemannian foliations, then $\left( M,\mathcal{F}\right) 
$ and $\left( M^{\prime },\mathcal{F}^{\prime }\right) $ have the same
codimension and dimension.
\end{proposition}

\begin{proof}
Every class in $H_{d}^{1}\left( M,\mathcal{F}\right) $ is represented by $%
f^{\ast }\theta $ for some closed basic one-form $\theta $ on $M^{\prime }$,
and $f^{\ast }$ is an isomorphism from $H_{d+\theta \wedge }^{k}(M^{\prime },%
\mathcal{F^{\prime }})$ to $H_{d+f^{\ast }\theta \wedge }^{k}(M,\mathcal{F})$%
. Since the largest $k$ such that $H_{d+\theta \wedge }^{k}(M^{\prime },%
\mathcal{F}^{\prime })\neq 0$ over all possible $\theta $ is the codimension
of $\mathcal{F}^{\prime }$ (with $\theta =-\kappa _{b}^{\prime }$), and the
codimension of $\mathcal{F}$ is computed similarly, the two codimensions of
the foliations must match. Since $f$ is in particular a homotopy equivalence
of the manifolds $M$ and $M^{\prime }$, $f^{\ast }$ induces isomorphisms on
ordinary cohomology of $M$, and therefore the dimensions of $M$ and $%
M^{\prime }$ are also the same. The result follows.
\end{proof}

We need the following results to prove connections between foliated homotopy
equivalences and cohomology.

\begin{proposition}
\label{PoincareDualityTwistedLichnProp}(Twisted Poincar\'{e} duality for
basic Lichnerowicz cohomology) If $\left( M,\mathcal{F}\right) $ is a
transversely oriented Riemannian foliation of codimension $q$ on a closed
manifold and $\theta $ is a closed basic one-form, then the transverse Hodge
star operator $\overline{\ast }:\Omega ^{k}\left( M,\mathcal{F}\right)
\rightarrow \Omega ^{q-k}\left( M,\mathcal{F}\right) $ induces the
isomorphism%
\begin{equation*}
H_{d-\theta \wedge }^{k}\left( M,\mathcal{F}\right) \cong H_{d-\left( \kappa
_{b}-\theta \right) \wedge }^{q-k}\left( M,\mathcal{F}\right).
\end{equation*}
\end{proposition}

\begin{proof}
From \cite{Al}, \cite{PaRi}, \cite{HabRi} we have the following identities
for operators acting on $\Omega ^{k}\left( M,\mathcal{F}\right) $:%
\begin{eqnarray*}
\overline{\ast }^{2} &=&\left( -1\right) ^{k\left( q-k\right) } \\
\left( \theta \lrcorner \right) &=&\left( -1\right) ^{q\left( k+1\right) }%
\overline{\ast }\left( \theta \wedge \right) \overline{\ast } \\
\delta _{b} &=&\left( -1\right) ^{q\left( k+1\right) +1}\overline{\ast }%
\left( d-\kappa _{b}\wedge \right) \overline{\ast } \\
\left( \theta \lrcorner \right) \overline{\ast } &=&\left( -1\right) ^{k}%
\overline{\ast }\left( \theta \wedge \right) \\
\overline{\ast }\left( \theta \lrcorner \right) &=&\left( -1\right)
^{k+1}\left( \theta \wedge \right) \overline{\ast } \\
\delta _{b}\overline{\ast } &=&\left( -1\right) ^{k+1}\overline{\ast }\left(
d-\kappa _{b}\wedge \right) \\
\overline{\ast }\delta _{b} &=&\left( -1\right) ^{k}\left( d-\kappa
_{b}\wedge \right) \overline{\ast }.
\end{eqnarray*}%
Then, letting the raised $\ast $ denote the $L^{2}$ adjoint with respect to
basic forms, 
\begin{equation*}
\left( d-\theta \wedge \right) ^{\ast }=\delta _{b}-\theta \lrcorner ,
\end{equation*}%
and the associated Laplacian is%
\begin{equation*}
\Delta _{\theta }=\left( \delta _{b}-\theta \lrcorner \right) \left(
d-\theta \wedge \right) +\left( d-\theta \wedge \right) \left( \delta
_{b}-\theta \lrcorner \right) .
\end{equation*}%
Then from the formulas above, if $\beta \in \Omega ^{k}\left( M,\mathcal{F}%
\right)$,%
\begin{eqnarray*}
\overline{\ast }\Delta _{\theta }\beta &=&\overline{\ast }\left( \delta
_{b}-\theta \lrcorner \right) \left( d-\theta \wedge \right) \beta +%
\overline{\ast }\left( d-\theta \wedge \right) \left( \delta _{b}-\theta
\lrcorner \right) \beta \\
&=&\left( -1\right) ^{k+1}\left( d-\kappa _{b}\wedge +\theta \wedge \right) 
\overline{\ast }\left( d-\theta \wedge \right) \beta \\
&&+\left( -1\right) ^{k}\left( \delta _{b}+\left( \theta -\kappa _{b}\right)
\lrcorner \right) \overline{\ast }\left( \delta _{b}-\theta \lrcorner
\right) \beta \\
&=&\left( -1\right) ^{k+1}\left( d-\kappa _{b}\wedge +\theta \wedge \right)
\left( -1\right) ^{k+1}\left( \delta _{b}-\left( \kappa _{b}-\theta \right)
\lrcorner \right) \overline{\ast }\beta \\
&&+\left( -1\right) ^{k}\left( \delta _{b}+\left( \theta -\kappa _{b}\right)
\lrcorner \right) \left( d-\kappa _{b}\wedge +\theta \wedge \right) 
\overline{\ast }\beta \\
&=&\Delta _{\kappa _{b}-\theta }\overline{\ast }\beta .
\end{eqnarray*}%
Thus, the operator $\overline{\ast }$ maps $\Delta _{\theta }$-harmonic
forms to $\Delta _{\kappa _{b}-\theta }$-harmonic forms and vice versa, so
from the basic Hodge theorem for basic Lichnerowicz cohomology (see \cite[%
Section 3.3]{IdaP}), $\overline{\ast }$ induces the required isomorphism.
\end{proof}

\begin{lemma}
(In \cite{Ban}, \cite{AitH}, \cite{IdaP}) \label{zeroCohomLemm}If $\left( M,%
\mathcal{F}\right) $ is Riemannian foliation of codimension $q$ on a closed,
connected manifold and $\theta $ is a closed basic one-form, then $%
H_{d-\theta \wedge }^{0}\left( M,\mathcal{F}\right) \cong \mathbb{R}$ if and
only if $\theta $ is exact. Otherwise, $H_{d-\theta \wedge }^{0}\left( M,%
\mathcal{F}\right) \cong \left\{ 0\right\} $.
\end{lemma}

\begin{proposition}
If $\left( M,\mathcal{F}\right) $ is a transversely oriented Riemannian
foliation of codimension $q$ on a closed, connected manifold and $\theta $
is a closed basic one-form, then $H_{d-\theta \wedge }^{q}\left( M,\mathcal{F%
}\right) \cong \mathbb{R}$ if and only if $\left[ \theta \right] =\left[
\kappa _{b}\right] \in H_{d}^{1}\left( M,\mathcal{F}\right) $.
\end{proposition}

\begin{proof}
By Proposition \ref{PoincareDualityTwistedLichnProp}, $H_{d-\theta \wedge
}^{q}\left( M,\mathcal{F}\right) \cong H_{d-\left( \kappa _{b}-\theta
\right) \wedge }^{0}\left( M,\mathcal{F}\right) $. By Lemma \ref%
{zeroCohomLemm}, this group is $\mathbb{R}$ if and only if $\kappa
_{b}-\theta $ is exact and is zero otherwise.
\end{proof}

It has been shown previously by H. Nozawa in \cite{Noz1}, \cite{Noz2} that
the \'Alvarez class $[\kappa_b]$ is continuous with respect to smooth
deformations of Riemannian foliations. The following proposition extends
these results further.

\begin{proposition}
\label{mean curv homo Prop}If $f:M\rightarrow M^{\prime }$ is a foliated
homotopy equivalence between transversely oriented Riemannian foliations
with basic mean curvatures $\kappa _{b}$ and $\kappa _{b}^{\prime }$,
respectively, then $\left[ f^{\ast }\kappa _{b}^{\prime }\right] =\left[
\kappa _{b}\right] \in H_{d}^{1}\left( M,\mathcal{F}\right) $.
\end{proposition}

\begin{proof}
By Proposition \ref{FolHomotpEquivIsoProp}, $f^{\ast }:H_{d-\kappa
_{b}^{\prime }\wedge }^{q}\left( M^{\prime },\mathcal{F}^{\prime }\right)
\rightarrow H_{d-f^{\ast }\kappa _{b}^{\prime }\wedge }^{q}\left( M,\mathcal{%
F}\right) $ is an isomorphism, so $H_{d-f^{\ast }\kappa _{b}^{\prime }\wedge
}^{q}\left( M,\mathcal{F}\right) \cong \mathbb{R}$ by the previous
proposition. Therefore, by the same proposition, $\left[ f^{\ast }\kappa
_{b}^{\prime }\right] =\left[ \kappa _{b}\right] \in H_{d}^{1}\left( M,%
\mathcal{F}\right) .$
\end{proof}

\begin{corollary}
\label{twistedBasicHomotopEquiCOr}If $f:M\rightarrow M^{\prime }$ is a
foliated homotopy equivalence between Riemannian foliations, then there
exist bundle-like metrics on $\left( M,\mathcal{F}\right) $ and $\left(
M^{\prime },\mathcal{F}^{\prime }\right) $ such that $f^{\ast }:\widetilde{H}%
^{k}\left( M^{\prime },\mathcal{F}^{\prime }\right) \rightarrow \widetilde{H}%
^{k}\left( M,\mathcal{F}\right) $ is defined and is an isomorphism.
\end{corollary}

\begin{proof}
By Proposition \ref{FolHomotpEquivIsoProp}, $f^{\ast }:\widetilde{H}%
^{k}\left( M^{\prime },\mathcal{F}^{\prime }\right) \rightarrow H_{d-\frac{1%
}{2}f^{\ast }\kappa _{b}^{\prime }\wedge }^{k}\left( M,\mathcal{F}\right) $
is an isomorphism, with $\theta ^{\prime }=\frac{1}{2}\kappa _{b}^{\prime }$%
. By Proposition \ref{mean curv homo Prop}, $\left[ f^{\ast }\kappa
_{b}^{\prime }\right] =\left[ \kappa _{b}\right] $. We then modify the
bundle-like metric so that $\kappa =\kappa _{b}=f^{\ast }\kappa ^{\prime }$
exactly by using \cite{Dom} to make the mean curvature basic and then by
multiplying the leafwise metric by a conformal factor to set the element of $%
\left[ \kappa _{b}\right] $.
\end{proof}

\noindent Recall that given a Riemannian foliation $(M,\mathcal{F})$, the
leafwise volume form $\chi_{\mathcal{F}}$ is related to the mean curvature $%
\kappa$ by Rummler's formula \cite{To} as $d\chi_{\mathcal{F}%
}=-\kappa\wedge\chi_{\mathcal{F}}+\varphi_0,$ where $\varphi_0$ is a $(p+1)$%
-form ($\mathrm{dim}\,\mathcal{F}=p)$ on $M$ such that $X_1\lrcorner\cdots%
\lrcorner (X_p\lrcorner \varphi_0)=0$ for all $X_1,\cdots, X_p\in \Gamma(L).$
We have

\begin{lemma}
Let $f:\left( M,\mathcal{F}\right) \rightarrow \left( M^{\prime },\mathcal{F}%
^{\prime }\right) $ be a foliated homotopy equivalence between Riemannian
foliations of codimension $2\ell $. Let $\nu $ and $\nu ^{\prime }$ be the
transverse volume forms for these metrics on $M$ and $M^{\prime }$,
respectively. Then there is a real nonzero constant $\lambda $ such that 
\begin{equation*}
\langle f^{\ast }(\alpha ),\nu \rangle _{M}=\lambda \frac{\mathrm{Vol}(M)}{%
\mathrm{Vol}(M^{\prime })}\langle \alpha ,\nu ^{\prime }\rangle _{M^{\prime
}}.
\end{equation*}%
for all $\alpha \in \Omega ^{2\ell }\left( M^{\prime },\mathcal{F}^{\prime
}\right) $.
\end{lemma}

\begin{proof}
We choose a metric on $\left( M^{\prime },\mathcal{F}^{\prime }\right) $
such that the mean curvature $\kappa ^{\prime }$ is basic. Then $\left[
f^{\ast }\kappa ^{\prime }\right] =\left[ \kappa _{b}\right] $ for any
metric on $M$ by Proposition \ref{mean curv homo Prop}. As in the last
proof, we then modify the bundle-like metric so that $\kappa =\kappa
_{b}=f^{\ast }\kappa ^{\prime }$ exactly. Since $\alpha $ is a basic $2\ell $%
-form, it defines a class in $H_{d-\kappa ^{\prime }\wedge }^{2\ell
}(M^{\prime },\mathcal{F^{\prime }})\simeq H_{b}^{0}(M^{\prime },\mathcal{%
F^{\prime }})\simeq \mathbb{R}[\nu ^{\prime }].$ We need to check that $[\nu
^{\prime }]\neq 0$ in the cohomology group $H_{d-\kappa ^{\prime }\wedge
}^{2\ell }\left( M^{\prime },\mathcal{F}^{\prime }\right) .$ Assume by
contradiction that it is zero, then there exists some basic form $\gamma \in
\Omega ^{2\ell -1}\left( M^{\prime },\mathcal{F}^{\prime }\right) $ of
degree $2\ell -1$ such that $\nu ^{\prime }=d\gamma -\kappa ^{\prime }\wedge
\gamma .$ Then, we compute 
\begin{eqnarray*}
\int_{M^{\prime }}\nu ^{\prime }\wedge \chi _{\mathcal{F}^{\prime }}
&=&\int_{M^{\prime }}d\gamma \wedge \chi _{\mathcal{F}^{\prime }}-\kappa
^{\prime }\wedge \gamma \wedge \chi _{\mathcal{F}^{\prime }} \\
&=&\int_{M^{\prime }}d\left( \gamma \wedge \chi _{\mathcal{F}^{\prime
}}\right) +\gamma \wedge d\chi _{\mathcal{F}^{\prime }}-\kappa ^{\prime
}\wedge \gamma \wedge \chi _{\mathcal{F}^{\prime }} \\
&=&\int_{M^{\prime }}-\gamma \wedge \kappa ^{\prime }\wedge \chi _{\mathcal{F%
}^{\prime }}-\kappa ^{\prime }\wedge \gamma \wedge \chi _{\mathcal{F}%
^{\prime }}=0.
\end{eqnarray*}%
This would imply $\mathrm{Vol}(M^{\prime })=0$, a contradiction. Therefore $%
[\nu ^{\prime }]\neq 0$. Hence $\alpha =\beta \nu ^{\prime }+(d-\kappa
^{\prime }\wedge )\phi $ for some $\beta \in \mathbb{R}$ and $\phi \in
\Omega ^{2\ell -1}\left( M^{\prime },\mathcal{F}^{\prime }\right) $. Thus, 
\begin{equation*}
\langle \alpha ,\nu ^{\prime }\rangle _{M^{\prime }}=\int_{M^{\prime
}}\alpha \wedge \chi _{\mathcal{F}^{\prime }}=\int_{M^{\prime }}(\beta \nu
^{\prime }\wedge \chi _{\mathcal{F}^{\prime }}+(d\phi -\kappa ^{\prime
}\wedge \phi )\wedge \chi _{\mathcal{F}^{\prime }}=\beta \mathrm{Vol}%
(M^{\prime }).
\end{equation*}%
The last term vanishes as a consequence of the previous computation (just
replace $\gamma $ by $\phi $). Also, we have 
\begin{eqnarray*}
\langle f^{\ast }(\alpha ),\nu \rangle _{M} &=&\int_{M}\beta f^{\ast }(\nu
^{\prime })\wedge \chi _{\mathcal{F}}+f^{\ast }(d\phi -\kappa ^{\prime
}\wedge \phi )\wedge \chi _{\mathcal{F}} \\
&=&\int_{M}\beta f^{\ast }(\nu ^{\prime })\wedge \chi _{\mathcal{F}%
}+(d(f^{\ast }\phi )-f^{\ast }(\kappa ^{\prime })\wedge f^{\ast }\phi
)\wedge \chi _{\mathcal{F}} \\
&=&\int_{M}\beta f^{\ast }(\nu ^{\prime })\wedge \chi _{\mathcal{F}}+\left(
d-\kappa \wedge \right) (f^{\ast }\phi )\wedge \chi _{\mathcal{F}%
}=\int_{M}\beta f^{\ast }(\nu ^{\prime })\wedge \chi _{\mathcal{F}}=\beta
\lambda \mathrm{Vol}(M).
\end{eqnarray*}%
In the last equality, we used the fact that $f^{\ast }\nu ^{\prime }$
defines a cohomology class $\left[ f^{\ast }\nu ^{\prime }\right] $ in $%
H_{d-\kappa \wedge }^{q}(M,\mathcal{F})\cong \mathbb{R}$ and therefore we
can write that $f^{\ast }\nu ^{\prime }=\lambda \nu +(d^{M}-\kappa \wedge
)\varphi $ for some real number $\lambda $. Furthermore, since $f$ is a
foliated homotopy equivalence, by Proposition \ref{FolHomotpEquivIsoProp},
it induces an isomorphism from $H_{d-\kappa ^{\prime }\wedge }^{q}(M^{\prime
},\mathcal{F}^{\prime })$ to $H_{d-\kappa \wedge }^{q}(M,\mathcal{F})$, so $%
\left[ f^{\ast }\nu ^{\prime }\right] $ is a nonzero class. Thus, the
constant $\lambda $ is nonzero.
\end{proof}

\begin{theorem}
\label{BasicSignHomotopyInvtThm}Let $f:M\rightarrow M^{\prime }$ be a
foliated homotopy equivalence between two Riemannian foliations of
codimension $2\ell $ and transverse volume forms $\nu $, $\nu ^{\prime }$
respectively. Then $\sigma (M,\mathcal{F})=\sigma (M^{\prime },\mathcal{%
F^{\prime }})$ if $f$ preserves the transverse orientation and $\sigma (M,%
\mathcal{F})=-\sigma (M^{\prime },\mathcal{F^{\prime }})$ otherwise.
\end{theorem}

\begin{proof}
We assume that $\ell $ is even, because the result is trivial when $\ell $
is odd. By Corollary \ref{twistedBasicHomotopEquiCOr}, we may choose metrics
such that $f^{\ast }:\widetilde{H}^{\ell }(M^{\prime },\mathcal{F^{\prime }}%
)\rightarrow \widetilde{H}^{\ell }(M,\mathcal{F})$ is an isomorphism and is
well-defined. For any $\alpha _{1}^{\prime },\alpha _{2}^{\prime }\in 
\widetilde{H}^{\ell }(M^{\prime },\mathcal{F}^{\prime })$, we compute 
\begin{eqnarray*}
A_{\mathcal{F}}(f^{\ast }\alpha _{1}^{\prime },f^{\ast }\alpha _{2}^{\prime
}) &=&\int_{M}f^{\ast }\alpha _{1}^{\prime }\wedge f^{\ast }\alpha
_{2}^{\prime }\wedge \ast \nu \\
&=&\langle f^{\ast }(\alpha _{1}^{\prime }\wedge \alpha _{2}^{\prime }),\nu
\rangle _{M} \\
&=&\lambda \frac{\mathrm{Vol}\left( M\right) }{\mathrm{Vol}(M^{\prime })}%
\langle \alpha _{1}^{\prime }\wedge \alpha _{2}^{\prime },\nu ^{\prime
}\rangle _{M^{\prime }}=\lambda \frac{\mathrm{Vol}\left( M\right) }{\mathrm{%
Vol}(M^{\prime })}A_{\mathcal{F}^{\prime }}(\alpha _{1}^{\prime },\alpha
_{2}^{\prime }).
\end{eqnarray*}%
Then by computing signatures of these quadratic forms, the conclusion
follows.
\end{proof}


\begin{thebibliography}{99}
\bibitem{AitH} H. Ait Haddou, \emph{Foliations and Lichnerowicz basic
cohomology}, Int. Math. Forum \textbf{2} (2007), no. 49-52, 2437--2446.

\bibitem{Al} J. A. \'{A}lvarez-L\'{o}pez, \emph{The basic component of the
mean curvature of Riemannian foliations}, Ann. Global Anal. Geom. \textbf{10}
(1992), 179--194.

\bibitem{ALMa} J.A. \'{A}lvarez L\'{o}pez, and X. Masa, \emph{Morphisms of
pseudogroups and foliated maps}, in Foliations 2005, World Sci. Publ.,
Hackensack, NJ, 2006, 1--19.

\bibitem{Ban} A. Banyaga, \emph{Some properties of locally conformal
symplectic structures}, Comment. Math. Helv. \textbf{77} (2002), no. 2,
383--398.

\bibitem{BenRA} M. Benameur and A. Rey-Alcantara, \emph{La signature basique
est un invariant d'homotopie feuillet\'{e}e}, C. R. Math. Acad. Sci. Paris 
\textbf{349} (2011), no. 13-14, 787--791.

\bibitem{BKRfol} J. Br\"{u}ning, F. W. Kamber, and K. Richardson, \emph{%
Index theory for basic Dirac operators on Riemannian foliations}, in
Noncommutative geometry and global analysis, Contemp. Math., \textbf{546},
Amer. Math. Soc., Providence, RI, 2011, pp. 39-81.

\bibitem{Dom} D. Dom\'{\i}nguez, \emph{Finiteness and tenseness theorems for
Riemannian foliations}, Amer. J. Math. \textbf{120} (1998), no. 6, 1237-1276.

\bibitem{EK} A. El Kacimi-Alaoui, \emph{Op\'{e}rateurs transversalement
elliptiques sur un feuilletage riemannien et applications}, Compositio Math. 
\textbf{73}(1990), 57--106.

\bibitem{EKN} A. El Kacimi-Alaoui and M. Nicolau, \emph{On the topological
invariance of the basic cohomology}, Math. Ann. \textbf{295} (1993), no. 4,
627--634.



















\bibitem{GlazKamb} J. Glazebrook and F. Kamber, \emph{Transversal Dirac
families in Riemannian foliations}, Comm. Math. Phys. \textbf{140} (1991),
217--240.

\bibitem{GorLott} A. Gorokhovesky and J. Lott, The index of a transverse
Dirac-type operator: the case of abelian Molino sheaf, J. Reine Angew. Math. 
\textbf{678} (2013), 125--162.

\bibitem{HabRi} G. Habib and K. Richardson, \emph{A brief note on the
spectrum of the basic Dirac operator}, Bull. London Math. Soc. \textbf{41}
(2009), 683-690.

\bibitem{HabRic2} G. Habib and K. Richardson, \emph{Modified differentials
and basic cohomology for Riemannian foliations}, J. Geom. Anal. \textbf{23}
(2013), no. 3, 1314--1342.

\bibitem{HalRyb} S. Haller and T. Rybicki, \emph{On the group of
diffeomorphisms preserving a locally conformal symplectic structure}, Ann.
Global Anal. Geom. \textbf{17} (1999), no. 5, 475--502.

\bibitem{IdaP} C. Ida and P. Popescu, \emph{On the stability of transverse
locally conformally symplectic structures}, in BSG Proc. \textbf{20 }(2013),
Balkan Soc. Geometers, Bucharest, 1--8.







\bibitem{KTduality} F. W. Kamber and Ph. Tondeur, \emph{Duality theorems for
foliations}, Transversal structure of foliations (Toulouse, 1982), Ast\'{e}%
risque \textbf{116} (1984), 108--116.

\bibitem{KT1} F. W. Kamber and Ph. Tondeur, \emph{De Rham-Hodge theory for
Riemannian foliations}, Math. Ann. \textbf{277} (1987), 415--431.










\bibitem{Mo} P. Molino, \emph{Riemannian foliations}, Progress in
Mathematics \textbf{73}, Birkh\"{a}user, Boston 1988.

\bibitem{Noz1} H. Nozawa, \emph{Rigidity of the \'{A}lvarez class},
Manuscripta Math. \textbf{132} (2010), no. 1-2, 257--271.

\bibitem{Noz2} H. Nozawa, \emph{Continuity of the \'{A}lvarez class under
deformations}, J. Reine Angew. Math. \textbf{673} (2012), 125--159.

\bibitem{OrnSle} L. Ornea and V. Slesar, \emph{Basic Morse-Novikov
cohomology for foliations}, Math. Z. \textbf{284} (2016), no. 1-2, 469--489.

\bibitem{PaRi} E. Park and K. Richardson, \emph{The basic Laplacian of a
Riemannian foliation}, Amer. J. Math. \textbf{118} (1996), 1249--1275.





\bibitem{Re1} B. L. Reinhart, \emph{Differential geometry of foliations: The
fundamental integrability problem}, Ergebnisse der Mathematik und ihrer
Grenzgebiete, vol. 99, Springer-Verlag, Berlin, 1983.








\bibitem{Re} B. Reinhart, \emph{Foliated manifolds with bundle-like metrics}%
, Ann. of Math. \textbf{69} (1959), 119-132.

\bibitem{To} Ph. Tondeur, \emph{Geometry of foliations}, Monographs in
Mathematics \textbf{90}, Birkh\"{a}user Verlag, Basel 1997.

\bibitem{Va} I. Vaisman, \emph{Remarkable operators and commutation formulas
on locally conformal K\"{a}hler manifolds}, Compositio Math. \textbf{40}
(1980), no. 3, 287--299.

\end{thebibliography}
\end{document}